\documentclass[11pt]{amsart} 
\usepackage{amsmath,amstext,amsfonts,amssymb,url} 

\newtheorem{theorem}{\bf Theorem}[section] 

\newtheorem{corol}[theorem]{\bf Corollary} 
\newtheorem{defi}[theorem]{\bf Definition} 
\newtheorem{lemma}[theorem]{\bf Lemma} 
\newtheorem{prop}[theorem]{\bf Proposition} 

\newcommand{\abs}[1]{\lvert#1\rvert} 
\newcommand{\bl}{{\noi$\bullet$\ }} 
\newcommand{\ctl}{\centerline} 
\newcommand{\noi}{{\noindent}} 
\DeclareMathOperator{\cl}{cl} 
\newcommand{\Equi}{\Longleftrightarrow} 
 
\newcommand{\la}{{\langle}} \newcommand{\ra}{{\rangle}} 
\newcommand{\lf}{{\lfloor}} \newcommand{\rf}{{\rfloor}} 
\newcommand{\nd}{{\text{ and }}} 
 
\DeclareMathOperator{\sgn}{sgn} 
\newcommand{\sm}{{\smallsetminus}} 
\newcommand{\stx}{\begin{smallmatrix}} \newcommand{\estx}{\end{smallmatrix}} 

\newcommand{\g}{{\gamma}} 
\newcommand{\Lb}{{\Lambda}}  

\newcommand{\A}{{\mathbb A}} 
\newcommand{\D}{{\mathbb D}} 
\newcommand{\E}{{\mathbb E}} 
\newcommand{\Z}{{\mathbb Z}}

\newcommand{\cB}{{\mathcal B}} 
\newcommand{\cC}{{\mathcal C}} 
\newcommand{\cI}{{\mathcal I}} 

\DeclareMathOperator{\End}{End} 

\begin{document} 

\title[Index Systems]{On the index system of well-rounded lattices} 

\author[J. Martinet]{Jacques Martinet\,(*)} 



\address{%
Universit\'e de Bordeaux, Institut de Math\'ematiques, 
\newline\indent 
351, cours de la Lib\'e\-ration, 33405 Talence cedex, France} 
\email{Jacques.Martinet@math.u-bordeaux1.fr} 

\begin{abstract} 
Let $\Lb$ be a lattice in an $n$-dimensional Euclidean space $E$ 
and let $\Lb'$ be a Minkowskian sublattice of $\Lb$, that is, 
a sublattice having a basis made 
of representatives for the Minkowski successive minima of~$\Lb$. 
We consider the set of possible quotients $\Lb/\Lb'$ 
which may exists in a given dimension or among not too large 
values of the index $[\Lb:\Lb']$, indeed $[\Lb:\Lb']\le 4$, 
or dimension $n\le 8$. 
\end{abstract} 

\maketitle 

\section{Introduction}\label{secintro} 

Extending a deformation argument used in \cite{M2} to prove 
Minkowski's theorem on successive minima (Theorem~2.6.8; 
1996 in the French edition), I proved in \cite{M1} that 
the sets of isomorphisms classes of quotients $\Lb/\Lb'$ 
for $\Lb'$ a Minkowskian sublattice of $\Lb$ 
are the same that those we obtain by restricting ourselves 
to a {\em well rounded} lattice~$\Lb$, that is a lattice, 
the minimal vectors of which span~$E$. For this reason, 
as in the title, we restrict ourselves to pairs $(\Lb,\Lb')$ 
of a well-rounded lattice~$\Lb$ and a sublattice~$\Lb'$ 
generated by minimal vectors of~$\Lb$. 

\medskip 

In this paper we consider thus the following problem: 
what is for a given dimension~$n$ the set of possible quotients 
$\Lb/\Lb'$ for a given $\Lb$ as above when $\Lb'$ runs through 
the set of all sublattices of $\Lb$ having a basis made 
with minimal vectors of~\hbox{$\Lb$\,?} 

Our results heavily rely 
on results obtained in \cite{M1} (which extends previous work by Watson, 
Ry\v skov and Zahareva) in dimensions up to~$8$ and in \cite{K-M-S} 
in dimension~$9$. One knows (\cite{M1}, theorem~1.7) 
that for $\Lb$, $\Lb'$ as above, the index $[\Lb:\Lb']$ is bounded 
from above by $\g_n^{n/2}$ ($\g_n$ is the {\em Hermite constant} 
for dimension~$n$), an inequality which in particular bounds 
the annihilator $d$ of $\Lb/\Lb'$. Then $\Lb$ is generated by a basis 
$\cB=(e_1,\dots,e_n)$ of $\Lb'$ together with a finite set 
of vectors $e=\dfrac{a_1e_1+\dots a_n e_n}d$, defining this way 
a $\Z/d\Z$-code, namely the code with codewords $(a_1,\dots,a_n)$. 
In the two papers mentioned above, all the codes which may occur 
in a dimension~$n\le 9$ are listed (and in particular all possible 
quotients $\Lb/\Lb'$). But the existence of two given structures 
does not imply that they can be realized by sublattices 
of a same lattice~$\Lb$: for instance cyclic and non-cyclic 
quotients of order~$4$ exist in dimension~$8$, 
but whatever the dimension, no lattice $\Lb$ may have both these 
quotients without having sublattices with quotients cyclic of 
order~$2$. 

\medskip 

The aim of this paper is to throw some light on the various
existing combinations, according to the definition below, 
in which $\Lb'$ runs through the set of lattices having a basis 
made of minimal vectors of a given lattice~$\Lb$: 

\begin{defi} {\rm 
Let $\Lb$ be a well-rounded lattice.} 
\begin{enumerate} 
{\rm 
\item 
The {\em maximal index} of $\Lb$ is 
$\imath(\Lb)=\max_{\Lb'}\,[\Lb:\Lb']$. 
\item 
The {\em index system of $\Lb$}, 
denoted by $\cI(\Lb)$, is the set of isomorphism classes 
of quotients $\Lb/\Lb'$.
\item 
We denote by $\cI_n$ the union of index systems $\cI(\Lb)$ 
with $\dim\Lb=n$.} 
\end{enumerate}\end{defi} 

When there is no risk of confusion, we shall write for short 
$4$, \hbox{$4\cdot 2$}, \hbox{$4\cdot 2^2$} to denote quotients 
isomorphic to $\Z/4\Z$, $\Z/4\Z\times\Z/2\Z$, $\Z/4\Z\times(\Z/2\Z)^2$, 
\linebreak 
res\-pectively; and the notation (see Subsection~\ref{subsecn8i8}) 

\smallskip 
\ctl{$\cI(\Lb_{75})=\{1,2,3,4,2^2,5,6,4\cdot 2,2^3\}$} 

\smallskip\noi 
means that for the the lattice $\Lb_{75}$, all structures 
up to order~$8$ except cyclic groups of order $7$ or~$8$ 
may be realized by convenient sublattices having a basis 
of minimal vectors . 

\medskip 

In Section~2 we recall some known results, mainly extracted from 
\cite{M1} and \cite{K-M-S}. Section~3 is devoted to dimension~$6$, 
index $\imath\le 3$ and related questions, 
Section~4 to dimension~$7$ and index~$4$, 
and Section~5 to dimension~$8$. 
Most of the constructions of lattices having a given 
index structure have been done using the {\em PARI-GP} package.

\section{Minimal classes and codes}\label{secgen} 

\subsection{General results} \label{subsecgenres} 
As usual, $S(\Lb)$ denotes the set of minimal vectors of the lattice 
$\Lb$, and we set $s(\Lb)=\frac12\,\abs{S(\Lb)}$. 
Minimal classes are the equivalence classes of lattices 
for the relation 
$$L\sim L'\Equi \exists\,u\in\End(E)\mid u(L)=u(L') 
\nd u(S(L))=u(S(L'))\,,$$ 
equipped with the ordering defined by 
$$\cC\prec\cC'\Equi \exists \Lb\in\cC,\,\exists\Lb'\in\cC'\mid 
S(\Lb)\subset S(\Lb')\,.$$ 
Clearly the index system of a lattice solely depends on its minimal 
class, and if $\cC\prec\cC'$, then the index structure of $\cC$ 
is a subset of that of~$\cC'$. Also, to an $n$-dimensional 
class $\cC$ we canonically attach its extension $\cC_{ext}$ 
to dimension~$n+1$, that of the lattices $\Lb\perp\Z$ 
for $\Lb\in\cC$ scaled to minimum~$1$. 
Clearly $\cI(\cC_{ext})=\cI(\cC)$. However these trivial extensions 
will be useful to construct some ``exotic'' index systems; 
compare \cite{M2} Section~3, or \cite{M-S}, Section~7. 

\medskip 

Given $n$, $d$, and a code $C$ over $\Z/d\Z$, among all minimal classes 
$\cC$ on which $C$ may be realized, if any, there exists a smallest 
one for the relation $\prec$, obtained using an averaging argument 
(\cite{M1}, Section~8; see also \cite{M-S}, Section~3 for a more 
general setting). 
Denote by {\em $m\le n$ the cardinality of the support of $C$}. 
The case of a binary code is easy: the smallest class is 
that of the lattices constructed by adjoining to $\Lb'=\Z^n$ 
the vectors of the form $e=\frac{\sum a_i e_i}2$. 
In this case the $e_i$ are pairwise orthogonal. 

\smallskip 

We now describe a notation we shall use for cyclic quotients. 
There is then a single vector $e=\frac{\sum a_i e_i}d$ to consider. 
We may assume that the $a_i$ are zero for $m<i\le n$, 
and (negating some $e_i$ if need be) that we have 
$1\le a_1\le\frac d2$ otherwise. We then denote by $m_1$ the number 
of subscripts $j$ such that $a_j=i$ and by $S'_i$ the set of vectors 
$e_j$ for which $a_j=i$; we have $m_i\ge 0$ and $\sum_i m_i=m$. 
For a lifting of $C$ to a pair $(\Lb,\Lb')$ (if any), 
the scalar products may be chosen to have constant values $x_i$ 
on $S_i$ and $y_{i,j}$ on $S_i\times S_j$ 
(no $y_{i,j}$ if $m_i$ or $m_j=0$ and no $x_i$ if $m_i\le 1$). 
We may moreover assume that $x_i=y_{i,j}=0$ if $d$ is even 
and $i=\frac d2$. 

\smallskip 

The complete description of codes is given in 
Table~11.1 of \cite{M1} for $n\le 8$ and in various 
tables of \cite{K-M-S} for $n=9$, for instance in Table~2 
for cyclic quotients, together with invariants relative 
to $\Lb$ and $\Lb'$: the kissing numbers $s,s'$ 
and the {\em perfection rank $r$ of $\Lb$} 
(the set of similarity classes of lattices in the smallest class 
$\cC$ depends on $\frac{n(n+1)}2-r$ parameters). 
Recall that a lattice (or a minimal class) with $r=\frac{n(n+1)}2$ 
is called {\em perfect}. Thus a perfect minimal class is the set 
of similarity classes of a perfect lattice.

\subsection{Calculation of index systems} 
\label{subseccalcul} 
Once we know the index system $\cI$ of such a class $\cC$ , 
we are sure that an index system which does not contain $\cI$ 
cannot be realized using the corresponding code, 
which strongly limits the search of smaller index systems. 
As for the index systems of classes $\cC'\succ\cC$, 
they are contained in those of the perfect classes containing $\cC$. 

\smallskip 

To calculate $\cI(\Lb)$ for a given lattice $\Lb$, we can use 
the naive algorithm, consisting in extracting systems of
$n$~independent vectors of $S(\Lb)$ and listing the corresponding 
quotients $\Lb/\Lb'$. Its complexity is roughly $\binom s n$. 
This is too large in case of the root lattice $\E_8$ 
($\binom s n=\binom{120}8$), but works otherwise up to dimension~$8$, 
since by a theorem of Watson ({Wat1}), 
we have $s\le 75$ if we exclude $\E_8$. Dutour Sikiri\'c has given 
in \cite{K-M-S}, Appendix~B, a more efficient algorithm, 
which notably allowed him to deal with a lattice having $(s,n)=(99,9)$.

\subsection{Specific lattices} \label{subsecspec} 
We list below a few results. 

\smallskip 

\bl\underbar{Root lattices}. These are the integral lattices 
which are generated by norm~$2$ vectors. They are orthogonal sums 
of the irreducible root lattices $\A_n,\,n\ge 1$, $\D_n,\,n\ge 4$, 
and $\E_n,\,n=6,7,8$. 

\begin{prop} 
The index systems of irreducible root lattices are as 
\linebreak 
follows: 
$\cI(\A_n)=\{1\}$; 
$\cI(\D_n)=\{1,2,\dots,2^{\lf\frac{n-1}2\rf}\}$; 
$\cI(\E_6)=\{1,2,3\}$; 
\newline 
$\cI(\E_7)=\{1,2,3,4,2^2,2^3\}$;
$\cI(\E_8)=\{1,2,3,4,2^2,5,6,4\cdot 2,2^3,3^2,2^4\}$. 
\end{prop} 

\begin{proof} 
The result for $\A_n$ is part of an old theorem 
of Korkine and Zolotareff; see \cite{M2}, Section~6.1. 
The other cases are dealt with using the classification 
of root systems; see \cite{K-M-S}, Appendix~A, for $\D_n$ 
and \cite{M1}, Section~6 for~$\E_n$. 
\end{proof} 

\bl\underbar{Perfect lattices}. 
The classification of perfect lattices is known 
up to dimension~$n=8$. Up to $n=7$, disregarding root 
lattices and a few lattices with maximal index $\imath\le 2$, 
we are left with index systems $\{1,2,3\}$ ($n=6,7$), 
and $\{1,2,3,4\}$ and $\{1,2,3,4,2^2\}$, 
this last one attained only on $P_7^{10}$ (in Conway and Sloane's 
notation $P_n^i$; see \cite{M2}, Section~6.5). 

In dimension~$8$, up to seven exceptions, the index systems share out 
among four types, namely 

\ctl{$\cI_4=\{1,2,3,4,2^2\},\ \cI_5=\cI_4\cup\{5\},\ 
\cI_6=\cI_4\cup\{6\},\,\nd\, \cI_{5,6}=\cI_4\cup\{5,6\}\,.$} 

\noi{\small 
[The other systems, all attained on lattices 
having a perfect hyperplane section with the same minimum, 
are $\{1,2,3\}$ (twice), $\{1,2,2^2\}$, 
and those of the the three irreducible root lattices 
and of Barnes's lattice $\A_8^2=\la\E_7,\A_8\ra$, 
for which $\cI=\cI(\E_7)$.]} 

\smallskip 

\bl\underbar{Maximal index systems}. The maximal index systems 
up to $n=9$ are classified in \cite{M1} and \cite{K-M-S}. 
They reduce to $\{1\}$ if $n=2,3$, attained on all lattices, 
and to $\{1,2\}$ if $n=5$, attained on a $9$-parameters family. 
For $n=$, $7$, $8$, they are attained uniquely on the perfect classes 
of $\D_4$, $\E_7$ and $\E_8$, respectively. 
For $n=6$ and $n=9$, there are two maximal systems. If $n=6$, 
these are $\{1,2,2^2\}$, attained on $\D_6$, and $\{1,2,3\}$, attained on 
on a $10$-parameters family. If $n=9$, they are attained 
on the perfect classes of the laminated lattice $\Lb_9$ and 
of a lattice denoted by $L_{81}$ in \cite{K-M-S}. 
\linebreak 
We have 

\ctl{$\cI_9= \{1,2,3,4,2^2,5,6,7,8,4\cdot 2,2^3,9,3^2,10,
12,6\cdot 2, 4^2,4\cdot 2^2,2^4\}\,,$} 

\smallskip\noi 
$\cI(\Lb_9)=\cI_9\sm\{4^2\}$, and $\cI(L_{81})$ 
consists of all structures of index up to $8$ 
and the three structures of order~$16$. 

\smallskip 

\bl\underbar{Watson's identity}. We consider the case when $\Lb/\Lb'$ 
is cyclic, writing $\Lb=\la\Lb',e\ra$ with 
$e=\dfrac{a_1 e_1+\dots+a_n e_n}d$. 
Denoting by $\sgn(x)$ the sign of the real number~$x$, we have 
the identity 

\smallskip 
\ctl{$\big((\sum_{i=1}^n\abs{a_i})-2d\big)N(e)=
\sum_{i=1}^n\abs{a_i}\big(N(e-\sgn(a_i)e_i)-N(e_i)\big)\,,$} 

\smallskip\noi 
which implies that when the $a_i$ are strictly positive and 
add to $2d$, all the vectors $e-e_i$ are minimal. 
In this case, $\cI(\Lb)$ contains $\{1,2,\dots,d\}$. 

\section{Maximal index 3, dimension 6, and bases versus generators}%
\label{secind3} 

In this section, we prove the classification results for lattices 
of maximal index~$\imath\le 3$ or dimension~$n\le 6$. 
We then consider some structures with $\imath=4$ corresponding 
to lattices generated by their minimal vectors which do not have 
any basis of minimal vectors. 

\smallskip 

We shall give a common proof for the two theorems stated below. 
In the second one, we only list the new structures, 
those which do not exist in a lower dimension. 

\begin{theorem} \label{thimax3} 
Let $\Lb$ be a lattice of maximal index $\imath\le 3$. 
The possible structure and the lower dimension $n_{\min}$ 
in which they exist are as follows: 

\noi{\small 
$\{1,2\}$\,: $n_{\min}=4$\,; 
$\{2\}$\,: $n_{\min}=5$\,; 
$\{1,2,3\}$\,: $n_{\min}=6$\,; 
\hbox{$\{3\},\{1,3\}$\,: $n_{\min}=7$\,;} 
$\{2,3\}$\,: $n_{\min}=11$.} 
\end{theorem}

\begin{theorem} \label{thnmax6} 
Let $\Lb$ be a lattice of dimension $n\le 6$. 
Then the minimal structures which exist in this dimension, but not in
a lower one, are as follows: 

\noi 
$n\le 3$\,: $\{1\}$\,; $n=4$\,: $\{1,2\}$\,; $n=5$\,: $\{2\}$\,; 
$n=6$\,: $\{1,2,3\},\{1,2,2^2\}$. 
\end{theorem} 

\begin{proof} 
By Watson's identity, $\imath=2$ (resp. $\imath=3$) is possible 
only if $n\ge 4$ (resp. $n\ge 6$), and if equality holds, 
the index structure must be $\{1,2\}$ (resp. $\{1,2,3\}$). 
That these lower bounds suffice can be seen 
in Table~11.1 of~\cite{M1}. This moreover shows that other systems 
of maximum index~$2$ (resp. $3$) need $n\ge 5$ (resp. $n\ge 7$). 

For any $n\ge 5$, taking $\Lb=\Lb'\cup(\frac{e_1+\dots+e_n}2+\Lb')$ 
with pairwise orthogonal vectors $e_i$, we obtain a lattice 
$\Lb$ with $S=S(\Lb')$ (and $s=n$). 

Let now $n=7$ and $\Lb=\Lb'\cup\pm(\frac{e_1+\dots+e_7}2)$ 
with equal scalar products $e_i\cdot e_j=x_1$. 
Then for $x_1=\frac 15$ (resp. $x_1=\frac 1{21}$), 
we obtain a lattice $\Lb$ with $S=S(\Lb')$ and $s=n$ 
(resp. $S=S(\Lb')\cup\{\pm e\}$ and $s=n+1$), 
and it is then evident that $\cI(\Lb)=\{3\}$ (resp. $\cI(\Lb)=\{1,3\}$). 

Finally it proved in \cite{M3}, Lemma~3.2, that $\cI=\{2,3\}$ 
needs $n\ge 11$, a lower bound which is optimal by a result 
of Conway and Sloane (\cite{C-S}). 

This completes the proof of Theorem \ref{thimax3}. 
\newline{\small 
(Nevertheless we shall sketch below a proof of the bound above 
and then adapt it to neighbour situations.)} 

\smallskip 

By \cite{M1}, Table~11.1, if $n\le 3$, $n=4$ or~$5$, $n=6$, 
we have $\imath=1$, $\imath=2$, $\imath=4$, respectively. 
As a consequence, the assertions of Theorem~\ref{thnmax6} result from 
the proof above of Theorem~\ref{thimax3}, 
except possibly for $n=6$ and $\imath>3$. 
But using again Table~11.1 of \cite{M1}, we see that this may occur 
only if $\imath=4$ and $\Lb\sim\D_6$, which implies that 
$\Lb/\Lb'$ is $2$-elementary. 
\end{proof} 

\noi{\em Proof of Theorem~\ref{thimax3} for $1\notin\cI$.} 
{\small 
We write as above 
$\Lb=\Lb'\cup(\pm e+\Lb')$ with $e=\frac{e_1+\dots+e_m}2$ 
for some $m\le n$. Since $\cI(\Lb)\supset_{\ne}\{1\}$, 
there is a minimal vector $x$ in $e+\Lb'$, say, 
$x=\frac{a_1 e_1+\dots+a_n e_n}3$. 
By Watson's identity for denominator~$3$, we have $m\ge 6$, 
and even $m\ge 7$ since otherwise $e-e_1$ would be minimal. 
Since $\imath(\Lb)=3$, we have $\abs{a_i}\le 3$, 
hence $a_i=1$ or~$-2$ if $i\le m$, and $a_i=0,\pm 3$ if $i>m$. 
Again because $1\notin\cI$, we must have $a_i=-2$ if $i\le m$, 
and choosing $n$ minimal, $a_i\ne 0$ if $i>m$. 
Watson's identity for denominator~$2$ now implies $1+n-m\ge 4$, 
and even $1+n-m\ge 5$ since otherwise 
$\frac{x+e_{m+1}+e_{m+2}+e_{m+3}}2$ would be minimal.} 
Hence $n\ge m+(n-m)\ge 7+4=11$. 
\newline{\small 
[Note that the averaging argument shows that three values suffice 
for the $e_i\cdot e_j$ ($x_1$ if $i<j\le m$, $x_2$ if $j>i>m$, 
$y_1$ if $i\le m<j$). 
The example of \cite{C-S} is constructed this way.]} \qed 

\smallskip 

The proof above applies directly to index systems $\{2,3,2^2\}$, 
and with slight modifications, to index systems containing $\{3,4\}$ 
but not $\{1,3,4\}$. However we may prove better results in some 
cases, but before analyzing more closely these lattices 
which are generated by their minimal vectors without having a basis 
of minimal vectors, we prove some lemmas. 

\begin{lemma} \label{lemmaxind} 
Assume $\Lb$ has maximal index $d$ and 
let $x=\frac{a_1 e_1+\dots+a_n e_n}d\in\Lb$. 
If $x$ is minimal, then $\abs{a_i}\le d$ for all~$i$. 
\end{lemma} 

\begin{proof} 
Let $\delta=\gcd(a_1,\dots,a_n)$, 
and set $a'_i=\frac{a_i}\delta$, $d'=\frac d\delta$, 
$\Lb''=\la\Lb',x\ra$, and $\Lb_0=\la x,e_j,j\ne i\ra$. 
We have $[\Lb'':\Lb']=d'$, hence $[\Lb:\Lb'']=\frac{d}{d'}$. 

Now let $i\in\{1,\dots,n\}$. If $a_i=0$, there is nothing to prove. 
Otherwise we may write 
$e_i=\frac{\sum_{i\ne i}\,a_j e_j-d' x}{\abs{a'_i}}$, with a denominator 
coprime with the $\gcd$ of the coefficients of the numerator. 
Hence we have $[\Lb'':\Lb_0]=\abs{a'_i}$, whence 
$[\Lb:\Lb_0]=\frac{d}{d'}\,\abs{a_i'}\le d$, 
i.e., $\abs{a'_i}\le d'$, and finally $\abs{a_i}\le d$. 
\end{proof} 

\begin{corol} \label{cormaxind} 
Let $\Lb$ be a lattice of maximal index $d$, and suppose that we have 
$\Lb=\la\Lb',f_1,\dots,f_\ell\ra$ with 
$f_k=\frac{\sum_i\,a_i^{(k)}e_i}{d_k}$ 
and that $d=d_1\cdots d_\ell$. 
If $x=\frac{\sum_i\,b_i e_i}{d_k}\in f_k+\Lb'$ is minimal, 
then $\abs{b_i}\le d_k$. 
\end{corol} 

\begin{proof} 
Set $\Lb''=\la\Lb',f_k\ra$. We have $[\Lb'':\Lb']=d_k$, 
hence $[\Lb:\Lb'']=\frac d{_k}$ and 
$\imath(\Lb'')\le\frac{\imath(\Lb)}{d/d_k}=d_k$, 
whence the result by Lemma~\ref{lemmaxind}. 
\end{proof} 

\begin{corol} \label{cor2elem} 
Suppose that $\Lb/\Lb'$ is $2$-elementary of order $2^\ell$ 
and that $\imath(\Lb)=2^\ell$. 
Then if $\cI(\Lb)$ strictly contains ${2^\ell}$, 
it contains $\{2^{\ell-1},2^\ell\}$. 
\end{corol} 

\begin{proof} 
The hypothesis shows that $S(\Lb)$ strictly contains $S(\Lb')$, 
hence that there exists $x\in S(\Lb)\sm\Lb'$. Since $\Lb/\Lb'$ 
is $2$-elementary, $x$ is of the form $\frac{\sum\,a_i e_i}2$. 
By Corollary~\ref{cormaxind}, we have $\abs{a_i}\le 2$, 
and since $x\notin\Lb'$, $a_i$ is odd for some subscript~$i$. 
Let $\Lb''=\la\Lb',x\ra$. Replacing $e_i$ by $x$ for such a subscript, 
we obtain a basis of minimal vectors for $\Lb''$, and since 
$\Lb/\Lb''$ is $2$-elementary, $2^{\ell-1}$ belongs to $\cI(\Lb)$. 
\end{proof} 

We now return to index systems for lattices of maximal index~$4$. 

\begin{prop} \label{propind34} 
\begin{enumerate} 
\item 
If $\cI(\Lb)=\{2,3,4\}$ or $\{2,3,4,2^2\}$, then $n\ge 11$. 
\item 
If $\cI(\Lb)=\{3,4\}$, then $n\ge 15$. 
\end{enumerate} 
\noi{\small\rm 
[By Corollary~\ref{cor2elem}, the index system $\{3,4,2^2\}$ 
does not exist.]} 
\end{prop} 

\begin{proof} 
We start as above with $\Lb=\la\Lb',e\ra$ and 
$$e=\frac{e_1+\dots+e_{m_1}+2(e_{m1+1}+\dots+e_m)}4= 
\frac{e'+e_{m1+1}+\dots+e_m}2$$
($m=m_1+m_2\le n$, $e'=\frac{e_1+\dots+e_{m_1}}2$). 
Since $3\in\cI(\Lb)$, $S(\Lb)$ is not contained in $S(\la\Lb',e'\ra$, 
so that there exists a minimal vector $x\in e+\Lb$, say, 
\linebreak $x=\frac{a_1 e_1+\dots+a_n e_n}4$. 
As above we have $a_i=-3$ if $i\le m_1$, $a_i=\pm 2$ if $m_1<i\le m$, 
$a_i=0,\pm 4$ if $i>m$, and indeed $a_i\ne 0$ if $n$ is minimal. 
Since $1\notin\cI$, $e'$ cannot be minimal, which implies $m_1\ge 5$. 
Using the denominator~$3$ provided by the $a_i$ with $i\le m$, 
we see that we must have $1+m_2+(n-m)\ge 7$, 
i.e., $n\ge m_1+6\ge 11$. 

If moreover $2\notin\cI$, we must have $m_2=0$, hence $m_1\ge 8$ 
by Watson's identity, and even $m_1\ge 9$ because $e-e_1$ cannot 
be minimal, hence finally $n\ge m_1+6\ge 15$. 
\end{proof} 

The case of an index system $\{2,3,2^2\}$ is slightly more 
complicated. 

\begin{prop} \label{propind32^2} 
If $\cI(\Lb)=\{2,3,2^2\}$ or $\{2,3,4,2^2\}$, then $n\ge 13$. 
\end{prop} 

\begin{proof} 
We denote by $e$, $f$, and $g\equiv e+f\mod\Lb'$ representatives 
of the non-zero cosets of $\Lb/\Lb'$, chosen so as to have components 
$0$ or $\frac 12$. Since $1\notin\cI(\Lb)$, the supports 
of $e$, $f$, $g$ have cardinality at least $5$ 
(i.e., the code has weight $w\ge 5$), which implies $m\ge 8$. 
We split $\{1,\dots,m\}$ into three sets 
$I,J,K$ such that $e=\frac{\sum_{i\in I\cup J}\,e_i}2$ 
and $f=\frac{\sum_{i\in I\cup K}\,e_i}2$. 
Note that at least two of these sets 
are non-empty, and that exchanging $e,f$, 
we may assume that $\abs{I}\ge 3$. 
Since the odd index $3$ occurs in the index 
system, two of the cosets above, say, those of $e$ and $f$, 
contain minimal vectors, say, 
$$x=\frac{a_1 e_1+\dots+a_n e_n}2\ \nd\ 
y=\frac{b_1 e_1+\dots+b_n e_n}2\,.$$ 
By Corollary~\ref{cormaxind}, 
we have $a_i=\pm 1$ on $I\cup J$, $b_i=\pm 1$ on $K\cup J$, 
and $a_i,b_i=0,\pm 2$ otherwise, and not $(0,0)$ (for $i>m$) 
if we choose $n$ minimal. 

Since $1\notin\cC$, the determinants 
$\abs{\stx a_i&a_j\\b_i&b_j\estx}$ may not be equal to $\pm 1$. 
This proves that $b_i,i\in I$ and $a_j,j\in K$ are non-zero: 
taking $i\in I$ and $j\in K$ if $K\ne\emptyset$ 
and $j\in J$ otherwise, we obtain the determinants 
$\abs{\stx\pm 1&a_k\\b_i&\pm 1\estx}$ 
and $\abs{\stx\pm 1&\pm 1\\b_i&\pm 1\estx}$. 

Now, negating if need be some $e_i$ with $i\in I\cup J$, then $y$, 
then some $e_i$ with $i\in K$, we may assume that 
$a_i=+1$ on $I\cup J$, $b_{i_0}=+2$ for some $i_0\in I$, 
and $b_i=+1$ on $K$. On $i_0$ and $j\in J$, we have the determinant 
$\abs{\stx 1&1\\b_0&\pm 1\estx}$, hence $b_j=-1$. 
If $K\ne\emptyset$, since determinants 
$\pm 5$ are excluded (because $\imath=4$), we must have 
first $a_j=+2$ on $K$, then $b_i=+2$ on $I$. If $K=\emptyset$, 
we prove that $b_i=+2$ on $I$ by using one index $j\in J$. 

Now we have 
$$2(x+y)=3g+\!\sum_{i=m+1}^n(a_i+b_i)e_i\,,\,\text{ hence }\, 
x+y+\sum_{i=m+1}^n(a_i+b_i)e_i\equiv 0\!\!\!\mod 3\,.$$ 
This shows that at least five terms $a_i+b_i$ must be non-zero, 
which implies $n\ge m+5\ge 13$ 
\small{(and $a_i=b_i\pm 2$, $a_i=\pm 4$ and $b_i=\pm 4$ 
is impossible if $4\notin\cI$)}. 
\end{proof}

\section{Maximal index 4 and dimension 7} \label{secind4} 

The study of index $4$ is organized as follows: we distinguish three 
types of index systems, those in which index~$4$ occurs 
with $4$ alone, $2^2$ alone, or both $4$ and~$2^2$. In each case 
one has to consider $8$~possible systems, corresponding 
to the eight subsets of $\{1,2,3\}$ (including $\emptyset$). 
We obtain this way $24$ {\em a priori} possible systems. 
However, Corollary~\ref{cor2elem} shows that systems 
which strictly contain $\{2^2\}$ must contain $\{2,2^2\}$, 
so that at least seven systems are impossible. 
We state this result as a proposition: 

\begin{prop} \label{prop4imp} 
The seven index systems 
$\{1,2^2\}$, $\{3,2^2\}$, $\{1,3,2^2\}$, 
\linebreak 
$\{4,2^2\}$, $\{1,4,2^2\}$, $\{3,4,2^2\}$, $\{1,3,4,2^2\}$. 
are impossible for a lattice of maxi\-mal index~$4$. \qed 
\end{prop} 

We know (Propositions \ref{propind34} and \ref{propind32^2}) 
that the remaining four index systems containing $3$ but not~$1$ 
(namely, $\{3,4\}$, $\{2,3,4\}$, $\{2,3,4,2^2\}$ and $\{2,3,2^2\}$) 
need \hbox{$n\ge 11$}, so that, in dimension~$7$, 
we are left with only $13$~systems, among which $\{1,2,2^2\}$ 
exists in dimension~$6$. 
We shall prove that all other systems except possibly 
$\{1,2,3,2^2\}$ do exist, and give the corresponding minimal dimension. 

\begin{prop} \label{prop4max} 
The 
systems $\{1,2,4,2^2\}$, $\{1,2,3,4\}$ 
and $\{1,2,3,4,2^2\}$ exist in dimension~$7$, 
and for any other system $\cI$ with $\imath=4$ and $4\in\cI$, 
we must have $(m_1,m_2)=(5,2)$ or $n\ge 8$. 
\end{prop} 

\begin{proof} 
We know by \cite{M1} that $4\in\cI$ implies $n\ge 8$ 
or $(m_1,m_2)=(4,3)$, $(5,2)$ or $(6,1)$. If $(m_1,m_2)=(4,3)$, 
the smallest class is obtained with pairwise orthogonal $e_i$, 
and we can check that we then have $\cI=\{1,2,4,2^2\}$. 
Similarly, if $(m_1,m_2)=(6,1)$, Watson's identity shows 
that $S$ contains the vectors $e-e_i$, hence that $\cI$ 
contains $\{1,2,3,4\}$, and the averaging argument, taking 
$x_1=\frac 15$ (the only possible choice) produces a lattice 
with $\cI=\{1,2,3,4\}$. In both cases, the only larger system 
is $\{1,2,3,4,2^2\}$, that we know to exist (with $\Lb=P_7^{10}$). 
\end{proof} 

\begin{prop} \label{prop134} 
It $\cI=\{1,4\}$ or $\{1,3,4\}$, then $n\ge 9$, and these two systems 
exist in dimension~$9$. 
\end{prop} 

\begin{proof} 
We keep the usual notation $\Lb$, $\Lb'$, $m_1$, $m_2$, $m=m_1+m_2$ and 
\linebreak 
$e=\frac{\sum_{i=1}^n\,a_i e_i}4$. Since $1\in\cI$, the coset $e+\Lb'$ 
contains a minimal vector~$x$. If $m_2>0$, the numerator of $x$ 
has a component $\pm 2$, which implies $2\in\cI$. 
We thus have $m_2=0$, hence $m=m_1\ge 8$, and if $m_1=8$, 
Watson's identity shows that the vectors $e-e_i,i\le m$ are minimal, 
hence that $\cI$ contains $\{1,2,3,4\}$. 
This proves the lower bounds $n\ge m=m_1\ge 9$. 

Take $m_2=0$ and consider the systems $S_1=S(\Lb')\cup\{\pm e\}$ 
and $S_2=S(\Lb')\cup\{\pm(e-e_1)\}$. 
It is easily checked that the index system of $S_1$ (resp. $S_2$) 
is $\{1,4\}$ (resp. $\{1,3,4\}$). 
Lattices with $n=m_1=9$ and $S=S_1$ or $S_2$ are constructed 
as follows. For $S_1$, take $e_i\cdot e_j=7/72$ for $1\le i<j\le 9$. 
For $S_2$, take $e_1\cdot e_j=9/40$ for $2\le j\le 9$ 
and $e_i\cdot e_j=7/40$ for $2\le i<j\le 9$. 
\end{proof} 

\begin{prop} \label{prop242^2} 
If $\cI=\{2^2\}$ or $\{2,4,2^2\}$, then $n\ge 8$, and these systems 
exist in dimension~$8$. 
\end{prop} 

\begin{proof} 
Since $2^2\in\cI$, both systems may be constructed with a binary 
code of weight $w\ge 4$. Since a word of weight~$4$ lifts 
to a $\D_4$-section, the index system of a lattice constructed 
with a code of weight~$4$ contains~$2$. Hence, if $\cI=\{2^2\}$, 
we must have $n\ge 8$, and this condition suffices, 
since there exists a (unique) binary code of length~$8$, 
dimension~$2$, and weight system $5^2\cdot 6$. 

\smallskip 

Consider now the system $\{2,4,2^2\}$, and suppose that $n=7$. 
Since $4\in\cI$, we can write $\Lb$ with $\Lb/\Lb'$ 
cyclic of order~$4$, and since $\cI$ contains~$2$ but not~$1$, 
$S(\Lb)$ spans its sublattice $\Lb''$ which contains 
$\Lb'$ to index~$2$. Since $2^2\in\cI$, and since every code 
of length~$7$ and weight $w\ge 4$ has a word of weight~$4$, 
$\Lb''$ has a $\D_4$-section. This shows that we may take $m_1=4$, 
hence $m_2=3$, but we know that the corresponding smallest class 
$\cC$ has then index system $\{1,2,4,2^2\}$, a contradiction. 

This proves that $n\ge 8$, and taking $m_1=m_2=4$ and pairwise 
orthogonal scalar products, we obtain a lattice with one quotient 
of type~$2^2$ and two of type~$(4)$. 
\newline{\small 
[The averaging arguments applied on the one hand to cyclic quotients 
of order~$4$ with $m_1=m_2=4$, and on the other hand 
to the binary code with weight system $(4\cdot 5\cdot 7)$ yield 
the same lattice, which accounts for the existence of quotients 
of both the types $(4)$ and~$2^2$.]} 
\end{proof} 

We are now able to give the complete list 
of index structure in dimension~$7$. 

\begin{theorem} \label{thn7} 
Let $\Lb$ be a lattice of dimension $7$. 
Then the minimal structures which exist in this dimension 
are as follows: 
\begin{enumerate} 
\item 
$\imath\le 2$: $\{1\}$, $\{1,2\}$, $\{2\}$. 
\item 
$\imath=3$: $\{1,2,3\}$, $\{3\}$, $\{1,3\}$. 
\item 
$\imath=4$, $4\in\cI$, $2^2\notin\cI$: 
$\{4\}$, $\{2,4\}$, $\{1,2,4\}$, $\{1,2,3,4\}$. 
\item 
$\imath=4$, $2^2\in\cI$, $4\notin\cI$: 
$\{2,2^2\}$, $\{1,2,2^2\}$, $\{1,2,3,2^2\}$. 
\item 
$\imath=4$, $4\in\cI$, $2^2\in\cI$: 
$\{1,2,4,2^2\}$, $\{1,2,3,4,2^2\}$. 
\item 
$\imath=8$: $\{1,2,3,4,2^2,2^3\}$, attained uniquely on 
the class of~$\E_7$. 
\end{enumerate}\end{theorem} 

\begin{proof} 
The case of $\imath>4$ results from \cite{M1}, 
and that of $\imath\le 3$ results from Theorems \ref{thnmax6} 
and~\ref{thimax3}. We are thus left with lattices 
of maximal index~$\imath=4$. 

\smallskip 

\bl $4\in\cI$ and $2^2\notin\cI$. Four out of the eight possible systems 
are excluded by Propositions \ref{prop134} and \ref{propind34}, 
and $\{1,2,3,4\}$ is known to exist by Proposition~\ref{prop4max}. 
We construct the three remaining structures using cyclic quotients 
with $(m_1,m_2)=(5,2)$ and the unique parameter $x_1$ 
(by averaging, we may choose $x_2=y_1=0$). With $x_1=3/20$, 
$x_1=1/4$, and any $x_1\in(3/20,1/4)$ (e.g., $x_1=1/5$), 
we obtain lattices with index systems $\{1,2,4\}$, $\{1,4\}$, 
and $\{4\}$, respectively. 

\smallskip 

\bl $2^2\in\cI$, $4\notin\cI$. Five out of eight possible systems 
are excluded by Propositions \ref{prop4imp}, \ref{prop242^2} 
and \ref{propind32^2}, so that we are left with the systems 
$\{1,2,2^2\}$ and $\{1,2,4,2^2\}$, which are known to exist 
by Theorem~\ref{thnmax6} and Proposition~\ref{prop4max}, 
and the system $\{1,2,3,2^2\}$. To construct an example having 
this index system, we observed that among perfect lattices, 
$4^2\in\cI$ holds only on $P_7^1=\E_7$ and $P_7^{10}$, 
and then $\cI$ contains both $4$ and $2^2$. 
This shows that minimal classes having the right system must lie below 
Voronoi paths connecting either of these two lattices. Among the eleven 
paths connecting two copies of $\E_7$, the one with $s=32$ 
proved convenient. A computation with {\em PARI-GP} showed that 
index $1$, $2$, $3$, $4$ appears $923766$, $21832$, $90$, 
and $6$ times, respectively, the last case only with 
an elementary quotient. In all cases the binary code (of length~$7$) 
is the code with weight system $4\cdot 5^2$. 
Here is a a Gram matrix belonging to this path 
(indeed, the eutactic one): 

\smallskip\ctl{\small 
$\left(\stx 4&2&2&-2&-2&-1&-1\\2&4&2&-2&-2&1&-2\\ 
2&2&4&0&0&-1&-2\\-2&-2&0&4&2&-1&0\\-2&-2&0&2&4&0&0\\ 
-1&1&-1&-1&0&4&-1\\ -1&-2&-2&0&0&-1&4\estx\right)\,.$} 

\smallskip 

\bl $4\in\cI$, $2^2\in\cI$. 
Six out of eight systems are excluded by Propositions~\ref{prop4imp}, 
\ref{prop242^2} and \ref{propind34}, 
and the remaining two systems exist by Proposition~\ref{prop4max}. 
\end{proof} 

\begin{theorem} \label{thi4} 
The seven index systems $\{1,2^2\}$, $\{3,2^2\}$, $\{1,3,2^2\}$, 
$\{4,2^2\}$, $\{1,4,2^2\}$, $\{3,4,2^2\}$ and $\{1,3,4,2^2\}$ 
do not exist. The other systems, except two for which existence 
is not known, are listed below together with the minimal dimension 
in which they exist: 

\smallskip 

\bl $n_{\min}=\ 6$: $\{1,2,2^2\}$. 

\bl $n_{\min}=\ 7$: $\{4\}$, $\{2,4\}$, $\{1,2,4\}$, $\{1,2,3,4\}$, 
                 $\{2,2^2\}$, $\{1,2,3,2^2\}$, 
\newline\phantom{\ \bl $n_{\min}=6$:\,}
$\{1,2,4,2^2\}$, $\{1,2,3,4,2^2\}$. 

\bl $n_{\min}=\ 8$: $\{2^2\}$, $\{2,4,2^2\}$. 

\bl $n_{\min}=\ 9$: $\{1,4\}$, $\{1,3,4\}$. 

\bl $n_{\min}=11$: $\{2,3,4\}$. 

\bl $n_{\min}=15$: $\{3,4\}$. 

\smallskip\noi 
For the systems $\{2,3,2^2\}$ and $\{2,3,4,2^2\}$, if any, 
we must have $n_{\min}\ge 13$. 
\end{theorem} 

\begin{proof} 
All the assertions above, are easy consequences of the results 
proved in this section and in the previous one, 
except those which concern $n_{\min}=11$ or~$15$, for which we must 
construct lattices having convenient sets of minimal vectors. 

An example for the index systems $\{2,3,4\}$, with $n=11$, 
(resp. $\{3,4\}$, with $n=15$) has been obtained 
taking $(m_1,m_2)=(5,6)$ (resp. $(9,0)$), and using three 
values for the scalar products $e_i\cdot e_j$, 
$x_1$ for $i<j\le m_1$, $x_2$ for $j>i>m_1$, 
and $y_1$ obtained as a function of $x_1,x_2$ for $i\le m_1,j>m_1$. 
One may then take $(x_1,x_2)=(\frac 19,\frac 14)$ 
(resp. ($\frac 19,\frac 19)$). 

We display below Gram matrices in the scale which make them 
integral and primitive, both with $s=n+1$ minimal vectors, 
as in the proof of Proposition~\ref{propind34}: 

{\small 
$$An11i234=\left(\stx 8600&1756&1756&1756&1756&2135&2135&2135&2135&2135&2135\\1756&1440&160&160&160&412&412&412&412&412&412\\1756&160&1440&160&160&412&412&412&412&412&412\\1756&160&160&1440&160&412&412&412&412&412&412\\1756&160&160&160&1440&412&412&412&412&412&412\\2135&412&412&412&412&1440&360&360&360&360&360\\2135&412&412&412&412&360&1440&360&360&360&360\\2135&412&412&412&412&360&360&1440&360&360&360\\2135&412&412&412&412&360&360&360&1440&360&360\\2135&412&412&412&412&360&360&360&360&1440&360\\2135&412&412&412&412&360&360&360&360&360&1440\estx\right)\,;$$} 

{\small 
$$An15i34=\left(\stx1836&816&816&816&816&816&816&816&816&819&819&819&819&819&819\\816&1728&192&192&192&192&192&192&192&364&364&364&364&364&364\\816&192&1728&192&192&192&192&192&192&364&364&364&364&364&364\\816&192&192&1728&192&192&192&192&192&364&364&364&364&364&364\\816&192&192&192&1728&192&192&192&192&364&364&364&364&364&364\\816&192&192&192&192&1728&192&192&192&364&364&364&364&364&364\\816&192&192&192&192&192&1728&192&192&364&364&364&364&364&364\\816&192&192&192&192&192&192&1728&192&364&364&364&364&364&364\\816&192&192&192&192&192&192&192&1728&364&364&364&364&364&364\\819&364&364&364&364&364&364&364&364&1728&144&144&144&144&144\\819&364&364&364&364&364&364&364&364&144&1728&144&144&144&144\\819&364&364&364&364&364&364&364&364&144&144&1728&144&144&144\\819&364&364&364&364&364&364&364&364&144&144&144&1728&144&144\\819&364&364&364&364&364&364&364&364&144&144&144&144&1728&144\\819&364&364&364&364&364&364&364&364&144&144&144&144&144&1728\estx\right)\,.$$}
\end{proof}

\section{Dimension 8} \label{secdim8} 

The list of structures with maximal index $\imath\le 4$ 
can be extracted from Theorem~\ref{thi4}. 
This list consists of the lattices listed in Theorem~\ref{thn7}, 
together with the two systems $\{2^2\}$ and $\{2,4,2^2\}$. 

\smallskip 

For larger indices, the possible co-existence of $2^2$ and $5$ 
causes difficulties, as in the case of $2^2$ and $3$. 
For this reason, the existence of the structure $\{1,2,3,2^2,5\}$ 
remains open, whereas all other cases have been settled.

\subsection{Maximal index 5} \label{subsecn8i5} 
For maximal index $5$, independently of the dimension, 
there are restrictions, as in Proposition~\ref{prop4imp}, 
obtained with the same kind of proof: 
a system which contains $\{2^2,5\}$ must contain $\{2,2^2,5\}$. 
There are also lower bounds better that $n\ge 8$ 
for some special systems, as in Proposition~\ref{prop134}, 
related to the fact that one of the invariants $m_1,m_2$ 
must be equal to $2$ or~$3$ if $n\le 9$, which implies that systems 
$\{1,5\}$ and $\{1,4,5\}$ do not exist if $n\le 10$, 
and more precisely, that if $n\le 10$, a system which strictly 
contains $\{5\}$ must contain $\{2,5\}$ or $\{3,5\}$. 
And we also know by \cite{M-S} that if $\cI\supset_{\ne}\{5\}$ 
and $1\notin\cI$, then $n\ge 10$; a $10$-dimensional example, 
with index system $\{2,3,4,5\}$, is given in \cite{M-S}. 

\smallskip 

In the general notation of \cite{M1} for index~$5$, 
the cosets of $\Lb/\Lb'$ 
are those of $\Lb'$, $\pm e+\Lb'$ and $\pm e'+\Lb'$, where 

\smallskip 
\ctl{ 
$e=\frac{e_1+\dots+e_{m_1}+2(e_{m_1+1}+\dots+e_m)}5\ \nd\ 
e'=\frac{2(e_1+\dots+e_{m_1})-(e_{m_1+1}+\dots+e_m)}5\equiv 2e\,,$} 

\smallskip\noi 
with $8\le m\le n$ and $m_1\ge m_2$. Here $n=m=8$, 
and $(m_1,m_2)$ must be equal to $(4,4)$, $(5,3)$ or $(6,4)$. 
The smallest minimal class attached to a pair $(m_1,m_2)$ 
is invariant under the action of $S_{m_1}\times S_{m_2}$ 
and can be constructed using three parameters $x_1$, $x_2$, $y_1$, 
namely the scalar products $e_i\cdot e_j$ for $i<j\le m_1$, 
for $m_1<i<j$ and for $i\le m_1,j>m_1$, respectively. 
The corresponding sets of minimal vectors 
(which have $s=16$, $s=8$, $s=16$) 
together with possible choices for the parameters 
(e.g., $(\frac 14,\frac 14,0)$, $(\frac 14,\frac 18,\frac 1{16})$, 
$(\frac 3{10},\frac 18,\frac 18)$) 
are given in \cite{M1}, and we easily deduce from these data 
that the index systems are $\{1,2,3,5\}$, $\{5\}$, 
and $\{1,2,3,4,5\}$, respectively. 

\begin{theorem} \label{thi5} 
The index system of an $8$-dimensional lattice of maximal index~$5$ 
is one of $\cI_1=\{1,2,3,4,2^2,5\}$, $\cI_2=\{1,2,3,4,5\}$, 
$\cI_3=\{1,2,3,5\}$, $\cI_4=\{1,2,4,5\}$, $\cI_5=\{1,2,5\}$, 
$\cI_6=\{5\}$, 
and maybe $\{1,2,3,2^2,5\}$. 
\end{theorem} 

\begin{proof} 
The proof will involve three steps: (1) the construction of more index 
systems; (2) the proof that an index system which strictly contains 
$\{5\}$ indeed contains $\{1,2,5\}$; (3) the proof that an index system 
$\cI$ with $3\notin\cI$ must be equal to $\{1,2,5\}$ 
or to $\{1,2,4,5\}$. 

\smallskip 

(1) Choose $(m_1,m_2)=(5,3)$. 
Taking $(x_1,x_2,y_1)=(\frac14,\frac 1{12},\frac 1{60})$, 
we obtain a lattice with $S=\{\pm e_i,\pm e\}$, hence $\cI=\{1,2,5\}$; 
taking $(x_1,x_2,y_1)=(\frac14,-\frac 1{12},\frac 1{12})$, 
we obtain a lattice with $s=30$ and $\cI=\{1,2,3,4,2^2,5\}$. 
{\small 
[This last index system (but no smaller system) occurs 
for numerous perfect lattices.]} 

There remains to construct a lattice with $\cI=\{1,2,4,5\}$. 
To this end we now use $5$~parameters, restricting $x_1$ to 
$i,j\le 4$ and $y_1$ to $i\le 4$, introducing $z_1=e_i\cdot e_5$ 
($i\le 4$), $z_2=e_5\cdot e_j$ ($j\ge 6$), then setting 
$z_2=\frac 14 x_1+\frac 12 x_2+y_1-\frac 23 z_1+\frac 7{48}$ 
to ensure $e-e_5\in S$. Taking 
$(x_1,z_1,x_2,y_1)= (\frac 15,\frac 14,\frac 18,\frac 1{16})$, 
we obtain a lattice with $S=\{\pm e_i,\pm(e-e_5)\}$, 
hence $\cI=\{1,2,4,5\}$. 

\smallskip 

(2) Let $\Lb$ be a lattice with $\cI\supset_{\ne}\{5\}$. 
There is nothing to prove if $(m_1,m_2)=(4,4)$ or $(6,2)$, 
and we may assume that $S\supset_{\ne} S(\Lb')$, hence that there exists 
a minimal vector $x\in e+\Lb'$ or $x'\in e'+\Lb'$. 
Then the coefficients $a_i$ (resp. $a'_i$) in the numerator 
of $x$ (resp. of $x'$) are $1$ or $-4$ if $i\le 5$ and $2$ or $-3$ 
if $i\ge 6$ (resp. $2$ or $-3$ if $i\le 5$ and $-1$ or $+4$ 
if $i\ge 6$). The coefficients $a'_i$, $i\le 5$ 
of an $x'\in S(e'+\Lb')$ cannot all be equal to $-3$, since an index 
$3$ would then exist in dimension~$4$. Thus $a'_i=2$ for some~$i$, 
which proves (2) if $S(e'+\Lb')\ne\emptyset$. 

If there exists $x\in S(e+\Lb')$ with all $a_1=-3$, 
Watson's identity with denominator~$3$ shows that $\Lb$ contains 
a vector $y=\frac{x+b_1 e_1+\dots+b)5 e_5}3$ with $b_i=\pm 1$ 
($b_i\equiv a_i\text{\,($=1$\,or\,$-4$)}\!\!\mod 3$. Then $y-b_1 e_1$ 
is a minimal vector in $e'+\Lb'$, a contradiction. 

\smallskip 

(3) The proof will be a consequence of the following lemma: 

\begin{lemma} \label{lemi5} 
If $3\notin\cI(\Lb)$, 
then $S(\Lb)\subset T=\{\pm e_1,\dots,\pm e_8\pm e,\pm e',\pm(e-e_i)\}$ 
for {\em one} index $i\le 5$. 
\end{lemma} 

Indeed, it is readily verified that $\cI(T)=\{1,2,4,5\}$, 
which excludes the structures $\{1,2,2^2,5\}$ and $\{1,2,2^2,4,5\}$. 

\smallskip 

This completes the proof of Theorem~\ref{thi5}. 
\end{proof} 

\noi{\em Proof of Lemma~\ref{lemi5}.} 
We successively consider the cosets of $0$, $e$ and $e'$ in $\Lb$ 
modulo $\Lb'$, using the notation $e,e',a_i,a'_i$ introduced 
in the proof of (2) above. 

(1) Because of the bound $\imath(\Lb)\le 5$, the components 
of the minimal vectors of $\Lb'$ on the $e_i$ must be $0$ or $\pm 1$. 
We must discard vectors of the form $e_1+e_2$ or $e_6+e_7$, 
since we could the write $e$ using $7$~independent vectors 
in its numerator; and a base change will show that using a vector 
of the form $e_1-e_2$, $e_6-e_7$ or $e_1\pm e_6$, we may 
define $\Lb$ with $(m_1,m_2)=(4,4)$ or $(6,2)$. Using this remark, 
we easily see that if there were in $S(\Lb')$ a sum $e_i\pm e_j\pm e_k$ 
with more than two components, then we could again express $e$ 
using less than $8$~vectors in its numerator. 
 
\smallskip 

(2) the minimal vectors $x\in e+\Lb'$ must have $a_i=1$ or $-4$ 
if $i\le 5$ and $a_i=2$ if $i>5$, and not three or more $a_i$ 
 equal to $-4$, since otherwise we would have an index $4$ 
in a dimension smaller than~$7$. 

If, say, $a_1=a_2=-4$, we have 

\smallskip\ctl{ 
$e=x+e_1+e_2=\frac{(-x+e_3+e_4+e_5)/2+e_6+e_7+e_8}2$\,,} 

\smallskip\noi 
which shows that $\frac{\pm x\pm e_3\pm e_4\pm e_5}2$ 
are minimal. Setting $y=\frac{x+e_3+e_4+e_5}2$, 
we then have 

\smallskip\ctl{ 
$y=\frac{e_6+e_7+e_8-2(e_1+e_2)+3(e_3+e_4+e_5}5$\,,} 

\smallskip\noi 
and an index~$3$ shows up. 

Finally, if, say, $e-e_1$ and $e-e_2$ are minimal, 
the identity 

\smallskip\ctl{ 
$e=\frac{-(e-e_1)-(e-e_2)+e_3+e_4+e_5+2(e_6+e_7+e_8)}5$} 

\smallskip\noi 
with $8$ independent vectors in the numerator shows the existence 
of an index~$3$. This proves that $S(e+\Lb')$ must be a subset 
of $\{\pm e,\pm(e-e_i)\}$ for {\em one} $i\in\{1,2,3,4,5\}$. 
 
\smallskip 

(3) We must have $a'_i=2$ for $i\le 5$, and if some $a'_i$ 
were equal to~$4$ for $i\ge 6$, then we would an index~$2$ 
in a dimension least than~$4$. This proves that $S(e'+\Lb')$ 
must be a subset of $\{\pm e'\}$. 
\qed 

Here is a Gram matrix ($n=8$, $\imath=5$, $s=9$) 
with $S=\{\pm e_i, \pm(e-e_5)\}$: 

{\small$$\left(\stx 
1404&534&534&534&702&697&697&697\\534&1200&240&240&300&75&75&75\\ 
534&240&1200&240&300&75&75&75\\534&240&240&1200&300&75&75&75\\ 
702&300&300&300&1200&185&185&185\\697&75&75&75&185&1200&150&150\\ 
697&75&75&75&185&150&1200&150\\697&75&75&75&185&150&150&1200 
\estx\right)\,.$$} 

\subsection{Maximal index 6} \label{subsecn8i6} 
Listing the various combinations of maximal 
\linebreak 
index~$6$, with or without $2^2$ and\,/\,or $5$, looks very 
complicated beyond~$n=8$, though the codes are known 
in all dimensions (\cite{K-M-S}, Section~6). Thus we restrict ourselves 
to dimension~$n=8$. 

\begin{theorem} \label{thi6} 
The index system of an $8$-dimensional lattice of maximal index~$6$ 
is one of the three systems

\ctl{$\cI_1=\{1,2,3,4,2^2,5,6\}$, $\cI_2=\{1,2,3,4,2^2,6\}$ 
or $\cI_3=\{2,4,2^2,6\}$\,.} 
\end{theorem} 

\begin{proof} 
In \cite{M1}, table~11.1, six types of maximal index~$6$ 
are listed, among which we must discard the third one, 
which only exists for the class of~$\E_8$. 
Using the data of this table, we can determine the index system 
of the smallest minimal class in each case. 
Here are the results for each remaining five sets $(m_1,m_2,m_3)$: 
$(4,3,1),(3,4,1)$\,: $\cI_1$\,; $(2,4,2),(4,2,2)$\,: $\cI_2$\,; 
$(3,3,2)$: $\cI_3$. 

This shows first that the three structures listed above exist, 
and next that a further structure, if any, 
must strictly contain $\cI_3$ and must be realized using 
$(m_1,m_2,m_3)=(3,3,2)$. 
To deal with this case, we introduce the notation 

\smallskip\ctl{\small 
$e=\frac{e_1+e_2+e_3+2(e_4+e_5+e_6)+3(e_7+e_8)}6$, 
$e'=\frac{e_1+e_2+e_3-e_4-e_5-e_6}3$, 
$e''=\frac{e_1+e_2+e_3+e_7+e_8}2$\,.} 

\smallskip\noi 
By Watson's identity, the $6$ vectors $e'-e_i$, $e'+e_j$, 
$i=1,2,3$, $j=4,5,6$ are minimal. For a sublattice $L$ of $E$ 
with $S(L)\subset\Lb'\cup(e'+\Lb')$, we have $[\Lb:L]=2$, 
hence $\cI(L)\subset\cI_3$. Hence a lattice $L$ 
with $\cI(L)\supset_{\ne}\cI_3$ must have a minimal vector $x$ 
off the cosets of $0$ and $e'$, and moreover $\cI(L)$ must contain 
and odd number. Then its minimal vectors generate $L$, 
so that by \cite{M3}, we have $1\in\cI(L)$. 
To prove the theorem, it suffices to show that $\cI(L)$ then also 
contains~$3$. This we now prove. 

\smallskip 

If $\pm x\in e+\Lb'$, let $x=\frac{a_1 e_1+\dots+a_8 e_8}6$.
For $i=7$ or~$8$, we have $\pm a_i\equiv 3\mod 6$, hence $a_i=\pm 3$, 
and the existence of an index~$3$ is clear. 

\smallskip 

Let now $x=\frac{a_1 e_1+\dots+a_8 e_8}2\in e''+\Lb'$. 
The $a_i$ are odd for $i=1,2,3,7,8$ and even for $i=4,5,6$. 
We first show that $a_i=\pm 1$ for $i=1,2,3$. 
We have $e_1=2x-a_2 e_2-\dots-a_8 e_8$, so that $e$ may be written 
on the independent vectors $x,e_2,\dots,e_8$ 
in the form $e=\pm\dfrac{2x+b_2 e_2+\dots+b_8 e_8}{6\abs{a_1}}$. 
Since the $\gcd$ of the coefficients of the numerator is $1$ or~$2$, 
$3\abs{a_1}$ is an index for $L$, which implies $3\abs{a_1}\le 6$, 
hence $a_1=\pm 1$, and similarly $a_2,a_3=\pm 1$. 

Permuting $e_1,e_2,e_3$ and negating $x$ if need be, we may assume 
that $a_1=a_2=+1$ and write 
$\pm e=\pm\dfrac{2x+b_3 e_3+\dots+b_8 e_8}6$ 
with $b_3=0$ or $-1$ as a combination of seven minimal vectors 
with denominator~$6$. Since index $6$ is not possible in 
dimension~$7$, all $b_i$ must be even, and in particular, 
$b_3$ must be zero. Now $x$ is a combination of six minimal vectors 
with denominator~$3$ and coprime coefficients in the numerator. 
Watson's identity for denominator~$3$ shows that $e+\Lb'$ contains
minimal vectors, and we are back to the first case. 
\end{proof}

\subsection{Maximal index 8} \label{subsecn8i8} 
We know from \cite{M1} that we have $\imath\le 8$ 
except on the class of $\E_8$ (see Section\ref{secintro}), 
where there exists elementary quotients $\Lb/\Lb'$ 
of order $9$ and $16$, 
and that cyclic quotients of order $7$ or~$8$ do not exist 
in dimension~$8$. 
{} 
Six codes for index~$8$ are listed in Table 11.1 of \cite{M1}, $n=8$. 
We denote the corresponding smallest minimal classes by $\cC_{8a}$ 
to $\cC_{8f}$, and by $\cC{8g}$ that of $\E_7\oplus\A_1$, 
which extends $\cl(\E_7)$ to $n=8$; the quotient $\Lb/\Lb'$ 
is of type $(4\cdot 2)$ in the first three cases, and $2$-elementary 
in the remaining four cases. 
The class $\cC_{8f}$ (with $(s,r)=(32,23)$) is that of the lattice 
$L_{32}$ which lifts the unique binary code having weight system 
$(4^3\cdot 5^4)$. The class $\cC_{8b}$ is a Voronoi path 
\hbox{$\E_8$\,---\,$\E_8$} (with $(s,r)=(75,35)$) 
discovered by Watson, along which lattices have an $\E_7$-section 
(and also a $\D_7$-section). 
The first three codes 
define quotients of type $4\cdot 2$, the remaining four elementary 
quotients. Averaging on codes for classes $\cC_{8a}$ and $\cC_{8e}$ 
yields isometric lattices, with $(s,r)=(48,32)$, 
hence $\cC_{8a}=\cC_{8e}$. 

\smallskip 

We display below Gram matrices $M_{32}$ for $L_{32}$ and $W_{75}$ 
for the eutactic lattice $\Lb_{75}$ lying on the Watson path; 
the basis for $L_{32}$ is $(e_1,e_2,e_3,e,e_5,f,e_7,g)$ where 
{\small 
$$e=\frac{e_1+e_2+e_3+e_4}2,\ f=\frac{e_3+e_4+e_5+e_6}2\ \nd\ 
g=\frac{e_2+e_4+e_6+e_7+e_8}2\,,$$} 
and $(e_1,\dots,e_8)$ is an orthogonal basis for $\Lb'$: 
{\small 
$$M32=\left(\stx 4&0&0&2&0&0&0&0\\0&4&0&2&0&0&0&2\\0&0&4&2&0&2&0&0\\2&2&2&4&0&2&0&2\\0&0&0&0&4&2&0&0\\0&0&2&2&2&4&0&2\\0&0&0&0&0&0&4&2\\0&2&0&2&0&2&2&5\estx\right)\,;\quad 
W75=\left(\stx 4&2&2&2&2&2&2&1\\2&4&0&0&0&2&0&2\\2&0&4&2&2&0&0&0\\2&0&2&4&2&0&0&0\\2&0&2&2&4&0&0&0\\2&2&0&0&0&4&0&0\\2&0&0&0&0&0&4&0\\1&2&0&0&0&0&0&4\estx\right)\,.$$} 

We shall prove the following result: 

\begin{theorem} \label{thi8} 
The index system of an $8$-dimensional lattice $\Lb$ 
with $\imath(\Lb)>6$ is one of the following five systems: 

\ctl{$\cI_1=\{1,2,3,4,2^2,5,6,4\cdot 2,2^3,3^2,2^4\}$, 
$\cI_2=\{1,2,3,4,2^2,5,6,4\cdot 2,2^3\}$\,,} 

\ctl{$\cI_3=\{1,2,3,4,2^2,4.2,2^3\}$, $\cI_4=\{1,2,3,4,2^2,2^3\}$\,,} 

\ctl{$\cI_5=\{1,2,2^2,2^3\}$, $\cI_6=\{2,4,2^2,2^3\}$\,.} 

\smallskip\noi 
All these systems exist, $\cI_1$, $\cI_5$, $\cI_6$ 
on unique minimal classes, that of $\E_8$, $\D_8$ and $L_{32}$, 
respectively, and $\cI_2$, $\cI_3$, $\cI_4$, on several classes. 
The system $\cI_2$ is that of the Watson path, $\cI_3$ that of 
$\cC_{8a}=\cC_{8e}e$, and $\cI_5$ that of $\E_7\perp\A_1$ and also 
of one well-defined class $\cC'_{8f}\succ\cC_{8f}$ 
with $(s,r)=(33,24)$. 
\end{theorem} 

\begin{proof} 
We first list the invariants $(s,r)$ and $\cI$ of the six smallest 
minimal classes related to the seven codes listed above: 

\noi$\cC_{8a}$: $(s,r)=(48,32)$, $\cI=\cI_3$ ($\cC_{8e}=\cC_{8a}$); 

\noi$\cC_{8b}$: $(s,r)=(75,35)$, $\cI=\cI_2$; 

\noi$\cC_{8c}$: $(s,r)=(120,36)$, $\cI=\cI_1$ ($\cC_{8c}=\cl(\E_8$); 

\noi$\cC_{8d}$: $(s,r)=(56,36)$, $\cI=\cI_5$ ($\cC_{8d}=\cl(\D_8$); 

\noi$\cC_{8f}$: $(s,r)=(32,23)$, $\cI=\cI_6$; 

\noi$\cC_{8g}$: $(s,r)=(64,29)$, $\cI=\cI_4$. 

\smallskip 

Here are a few comments on the list above. From \cite{M1}, 
we know $\cI_1$ and $\cI_4$, and the fact that we have $\imath\le 8$ 
except on the class of $\E_8$. This shows that every index system 
except $\cI_1$ is contained in $\cI_2$. 

We also know that $\cC_{8b}$ contains $\cI(\E_7)$ and $\{4\cdot 2\}$. 
A computer search then quickly finds cyclic quotients of order $5$ 
and~$6$ (a few days computations finds the number of occurrences 
of all quotients $\Lb/\Lb'$), and the remaining calculations are much shorter. 
This proves the existence of the six index systems 
of Theorem~\ref{thi8}. Note also that that $4$ belongs to all systems 
except $\cI_5$. We have thus also proved the uniqueness assertions 
about $\cI_1$ and~$\cI_5$. 

\smallskip 

To classify all index systems containing $4\cdot 2$, 
it now suffices to consider classes containing $\cC_{8a}$. 
The perfection co-rank of $\cC_{8a}$ is sufficiently small 
($36-32=4$) to allows us to find all classes $\cC$ lying 
above $\cC_{8a}$ (in other words, to find its {\em Ryshkov polyhedron} 
in the sense of \cite{K-M-S}, Section~3). 
One class has $(s,r)=(49,33)$ and $\cI=\cI_3$. 
the other classes all have $\cI=\cI_2$ except the maximal one 
which is that of $\E_8$. (These have invariants $(33,56)$, 
$(34,57)$, $(35,66)$, and $(35,75)$, the Watson path.) 

This proves that the index systems containing $4\cdot 2$ 
are $\cI_1$, $\cI_2$, and $\cI_3$. 

\smallskip 

We now turn classes which lie above $\cC_{8g}$ and have an index 
system which strictly contains $\cI(\cC_{8g})=\cI_4$. The maximal 
classes are those of a perfect lattice with $\imath\ge 8$. 
The only such lattice is $\E_8$. By the results of \cite{D-S-V}, 
classes with $r\le 35$ are contained in the Watson path. 
While classifying the possible values of $s$ in dimension~$8$, 
the authors of \cite{D-S-V} have proved that there exists 
a unique class with $r=34$ lying below $\cC_{75}=\cC_{8b}$, 
which has $s=69$. I have checked that this class $\cC_{69}$ 
has again index system $\cI_2$. 

Now we have $s-r=35$ on $\cC_{69}$ as on $\cC_{8g}$. This shows 
that the classes $\cC$ such that $\cC_{8g}\prec\cC\prec\cC_{69}$ 
are obtained by removing arbitrary vectors off $\E_7$ 
from $\cC_{69}$. Testing equivalence, we have shown 
that there are two such classes with $r=33$, 
and have checked that all have $\cI=\cI_4$. 

This proves that the classes containing $\cC{8g}$ have index system 
$\cI_1$, $\cI_2$ or~$\cI_4$. 

\smallskip 

Finally we are left with classes $\cC\succ\cC_{8f}$. Classifying 
all possible classes is certainly complicated, since the minimal class 
$\cC_{8f}$ depends on $36-23=13$ parameters, namely the scalar products 
$e_i\cdot e_j$ for $i=1,\dots,6$ and $j=7,8$, and $e_7\cdot e_8$. 
Thanks to Lemma~\ref{lemi8} below, we can avoid such a classification. 
The matrix $M_{32}$ is obtained taking these parameters all zero. 
Replacing $0$ by $-\frac 1{12}$ for $i=2,4,6$ and $j=7,8$, 
we obtain a lattice with $s=33$, $r=24$ and $\cI=\cI4$. 
Its minimal class is the class $\cC'_{8f}$. 

\begin{lemma} \label{lemi8} 
Let $\cC$ be a minimal class containing {\em strictly} $\cC_{8f}$. 
Then one of the following assertions holds: 
\begin{enumerate} 
\item 
$\cC\succ\cC_{8g}$. 
\item 
$\cI(\cC)\supset\cI_2$. 
\item 
$\cC\succ\cC'_{8f}$, and $\cC$ can be defined by a set of minimal 
vectors contained in 
$S(\cC_{8f})\cup\{\frac{\pm e_2\pm e_4\pm e_6\pm e_7\pm e_8}2\}$. 
\end{enumerate} 
In all cases, $\cI(\cC)$ contains $\cI_4$. 
\end{lemma} 

Taking for granted this lemma, we can now complete the proof 
of Theorem~\ref{thi8}. The last assertion of the lemma 
proves that only $\cC_{8f}$ has index system $\cI_6$. 
Next a computer calculation on the few systems of minimal vectors 
as in (3) shows that either $\cI(\cC)=\cI_4$ or $\cI(\cC)$ 
contains~$\cI_2$. 
\end{proof} 

\noi{\em Proof of Lemma~\ref{lemi8}.} Let $\cC\succ_{\ne}\cC_{8g}$, 
let $\Lb\in\cC$, and let $x\in S(\Lb)\sm S(\Lb_{32})$, 
belonging to a coset $v+\Lb'$. We consider three cases: 
\newline 
(1) $v=0$ (i.e., $x\in\Lb'$); 
\newline 
(2) $v$ lifts a word of weight~$4$ (i,e., $v=e,f$ or $e+f$); 
\newline 
(3) $v$ lifts a word of weight~$5$ (i,e., $v=g,g+e,g+f$ or $g+e+f$). 
\newline 
Taking into account the automorphisms of the code, 
we may assume that $v=e$ in case~(2) and $v=g$ in case~(3). 
In all cases, by Corollary~\ref{cormaxind}, the components 
of $x$ on the basis $(e_1,\dots,g)$ used to construct $\Lb_{32}$ 
are $0,\pm 1$. 

\smallskip\noi 
(1) Let $x=\pm e_{i_1}\pm\dots\pm e_{i_k}$ (with one or two 
terms in $\{e_7,e_8\}$, since $e_i\cdot e_j=0$ if $i<j<7$). 
If $k=2$, replacing $e_7$ or $e_8$ by $x$ amounts to change 
the code into a code of length~$7$ generated by weight-$4$ words, 
that is, the code of~$\E_7$. There cannot be three components in
the support of a weight-$5$ word, and if, say, $x=e_1+e_2+e_7$, 
then $e_1+e_7$ is minimal since $e_1\cdot e_2=0$. 
The case when $k\ge 4$ similarly reduces to $k-1$. 

\smallskip\noi 
(2) We may assume using change of signs that 
$x=e\pm e_7$ or $e\pm e_7\pm e_8$. Replacing $e_1,e_2,e_3,e_4$ 
by the four vectors $\frac{e_1\pm e_2\pm e_3\pm e_4}2$ 
having $0$ or $2$ minus signs, we are back to the previous case. 

\smallskip\noi 
(3) We have $v=g$, so that the minimal class of $\la\Lb,x\ra$ 
is either $\cC'_{8g}$, or $x$ may be assumed to be equal to 
$g+e_1$, $g+e_1+e_3$, or $g+e_1+e_3+e_5$. The last two cases reduce 
to $x=g+e_1$ (because $e_1\cdot e_3=0$), and a computer calculation 
shows that $\cI$ then contains~$\cI_2$. 

Moreover, if $\cC$ contains besides $g$ a vector $y\ne\pm g$, 
then either $y$ belongs to $g+\Lb'$, and then $y$ is a vector $g'$ 
obtained from $g$ by changing signs of some $e_i$, 
or $y$ belongs to $e+g$, say, and using the argument used deal 
with~(2), we again reduce ourselves to the previous situation. 
\qed


\end{document}